\documentclass[10pt,twoside]{article}
\pdfoutput=1
\usepackage{amssymb}
\usepackage{amsmath}
\usepackage{latexsym}
\newcommand{\lz}[1]{\log_2{#1}}
\newcommand{\lr}[1]{\log_r{#1}}
\newcommand{\lqq}[1]{\log_{1/q}{#1}}

\newcommand{\rve}{\mathcal{RV}(1)}
\newcommand{\rk}{r^{k}}
\newcommand{\rkm}{r^{k-1}}

\newcommand{\qkm}{q^{k-1}}
\newcommand{\subj}[2]{\textsf{AMS 2000 subject classifications.}
Primary {#1}; Secondary {#2}.\newline}
\newcommand{\key}[1]{\textsf{Keywords and phrases.} {#1}.\newline}
\newcommand{\abb}[1]{\textsf{Abbreviated title.} {#1}.}

\newcommand{\fot}[5]{\renewcommand\thefootnote{}
\footnotetext{\parindent=0.0mm \vskip-3mm \subj{#1}{#2}\key{#3}\abb{#4}
\newline\textsf{Date.} \date{\today}}}

\newtheorem{theorem}{Theorem}[section]
\newtheorem{lemma}{Lemma}[section]

\newtheorem{remark}{\normalfont\scshape Remark}[section]

\newenvironment{proof}{\noindent\textsc{Proof.\/}}{}
\def\rrm{\normalfont\rmfamily}
\newcommand{\Exp}{\,\mbox{\rrm Exp}}
\def\vsb{\hfill$\Box$}
\def\vsp{\vskip-8mm\hfill$\Box$\vskip3mm}
\newcommand{\be}{\begin{equation}}
\newcommand{\ee}{\end{equation}}
\newcommand{\bea}{\begin{eqnarray}}
\newcommand{\eea}{\end{eqnarray}}
\newcommand{\beaa}{\begin{eqnarray*}}
\newcommand{\eeaa}{\end{eqnarray*}}
\newcommand{\var}{\mathrm{Var\,}}
\newcommand{\ttt}[1]{\quad\mbox{ #1}\quad}

\newcommand{\sumk}{\sum^n_{k=1}}
\newcommand{\sumin}{\sum_{n=1}^\infty}

\newcommand{\asto}{\stackrel{a.s.}{\to}}
\newcommand{\pto}{\stackrel{p}{\to}}

\newcommand{\dto}{\stackrel{d}{\to}}
\newcommand{\nifi}{n\to\infty}

\newcommand{\iid}{i.i.d.\ }
\newcommand{\xx}{X_1,\, X_2,\,\dots}

\begin{document}
\date{}
\title{\textsf{Generalized St.\ Petersburg games revisited}}
\author{Allan Gut\\Uppsala University \and 
Anders Martin-L\"of\\Stockholm University}
\maketitle

\begin{abstract}
\noindent
The topic of the present paper is a generalized St.\ Petersburg game
in which the distribution of the payoff $X$ is given by
$P(X=s\rkm)=pq^{k-1}$, $k=1,2,\ldots$, where $p+q=1$, and $s,\,r>0$.
As for main results, we first extend Feller's classical weak law and
Martin-L\"of's 1985-theorem on convergence in distribution along the
$2^n$-subsequence. In his 2008-paper Martin-L\"of considers a
truncated version of the game and the problem ``How much does one gain
until 'game over'\,'', and a variation where the player can borrow money
but has to pay interest on the capital, also for the classical
setting. We extend these problems to our more general setting. We
close with some additional results and remarks.
\end{abstract}

\fot{60F05, 60G50}{26A12}{St.\ Petersburg game, sums of \iid random
variables, Feller WLLN, convergence along subsequences}{Generalized 
St.\ Petersburg game}

\section{Introduction}
\markboth{Allan Gut and Anders Martin-L\"of}{Generalized St.\ Petersburg game}
\label{intro} \setcounter{equation}{0} 
The classical St.\ Petersburg game is defined as follows: Peter throws
a fair coin repeatedly until heads turns up. If this happens at trial
number $k$ he has to pay Paul $2^k$ ducates. The question is what the
value of the game might be to Paul. Now, since the random variable $X$
describing the payoff is governed by
\[P(X=2^k)=\frac1{2^k},\quad k=1,2,\ldots\,,\]
which has infinite expectation, we have no guidance there for what a
fair price would be for Paul to participate in the game.

One variation is to set the fee as a function of the number
of games, which leads to the celebrated Feller solution
\cite{feller45}, namely, that if $X,\,\xx$ are \iid random variables
as above, and $S_n=\sumk X_k$, $n\geq1$, then 
\bea\label{feller} \frac{S_n}{n\lz{n}}\pto1\ttt{as}\nifi, \eea 
where, generally, $\lr{(\cdot)}$ denotes the logarithm relative to
base $r>0$. For details, see \cite{FI}, Chapter X, and \cite{FII},
Chapter VII (and/or \cite{g13}, Section 6.4.1). More on the history 
of the game can be found in \cite{aml85}.

The present paper is devoted to the generalization in which I toss a
biased coin for which $P(\mbox{heads})=p$, $0<p<1$, repeatedly until
heads appears. If this happens at trial number $k$ you receive $s\rkm$
Euro, where $s,\,r>0$, which induces the random variable
\bea
\label{jag}P(X=s\rkm)=pq^{k-1},\quad k=1,2,\ldots.
\eea 
Our first result is an extension of Feller's weak law (\ref{feller})
to the setting (\ref{jag}) under the assumption that $r=1/q$. If, in
addition, $s=1/p$ the result reduces to Theorem 2.1(i) of \cite{g10},
where additional references can be found. The case, $p=q=1/2$
corresponds (of course) to the classical game.

As for convergence in distribution in the classical case, Martin-L\"of
\cite{aml85} obtains convergence in distribution along the
\emph{geometric subsequence\/} $2^n$ to an infinitely divisible,
semistable Our second result extends his theorem to the general case.

If, in particular, $s=1/p$ och $r=1/q$, (some of) the results reduce
to those of \cite{g10, g10c}, and if, in addition, $p=q=1/2$ to the
setting in \cite{aml85, aml08}.

The results mentioned so far are stated in Section \ref{wlln} and
proved in Sections \ref{pfthm1} and \ref{pfthm2}, respectively, after
some preliminaries in Section \ref{prel}. 

In Section \ref{stopp} we consider a truncated game and the problem
``How much does one gain until game over?'', thereby extending the classical
setting from \cite{aml08}. A second model treated in the cited paper
concerns the case when the player can borrow money without limit for
the stakes, but has to pay interest on the capital. Our extensions to
the present setting is treated in Section \ref{rta}.

 A final section contains some additional results and remarks.

We close the introduction by mentioning that some of our results
exist (essentially) as special cases of more general results. An important 
point here is that we provide more elementary and transparent proofs.

\section{Main results}
\label{wlln} \setcounter{equation}{0} 
Thus, let throughout $X,\,\xx$ be \iid random variables with
\[P(X=s\rkm)=pq^{k-1},\quad k=1,2,\ldots,\]
and set $S_n=\sumk X_k$ and
$M_n=\max_{1\leq k\leq n}X_k$, $n\geq1$.

Since we are aiming at weak limits we begin by noticing that  
if  $r<1/q$, then $E\,X<\infty$, so that the 
classical strong law holds, viz.\ 
\[\frac{S_n}{n}\asto \frac{s p}{1-rq}\ttt{as}\nifi,\]
where for the value of the limit we refer to (\ref{mombeta}) with $\beta=1$ 
below.

\emph{In the following we therefore assume that $rq\geq1$, and thus,
  in particular, that $r>1$.}
\begin{remark}\emph{If, in addition, $r<\sqrt{q}$, then $\var X<\infty$ 
and a central limit theorem holds.}\vsb\end{remark}

\begin{theorem}\label{thm1} If $r=1/q$, then
\[\frac{S_n}{n\lr{n}}\pto sp\ttt{as}\nifi.\]
\end{theorem}

\begin{remark}\emph{For $p=q=1/2$ and $\alpha=1$ the theorem reduces to 
(\ref{feller}), and for general $p$, $s=1/p$, and $\alpha=1$ to
    \cite{g10}, Theorem 2.1(i).  For the case $s=r=q^{-1}$, Adler and
Rosalsky, \cite{adros}, Theorem 4, prove a weak law for the 
weighted sum $\sumk k^\gamma X_k$, where $\gamma>-1$.}\vsb
\end{remark}

Our next theorem extends Martin-L\"of's subsequence result for the
classical game \cite{aml85}. We remark that $M$, $N$, and $uM$ below
are not integers. We leave it to the reader to replace such quantities
with the respective integer parts and to make the necessary amendments.

\begin{theorem} \label{thm2} Let $N=r^n$ and $M=q^{-n}$.\\
\noindent \emph{(i)} If $r=1/q$, then, for $u>0$,
\[\frac{S_{uN}-spuNn}{N}=\frac{S_{uN}}{N}-spun\dto Z(u)\ttt{as}\nifi\,, \]
where $Z(u)$ is the L\'evy process defined via the characteristic function
$\varphi_{Z(u)}(t)=E\exp\{itZ(u)\}=\exp\{ug(t)\}$, where
\beaa 
g(t)&=&\sum_{k=-\infty}^{-1}
\big(\exp\{itsr^k\}-1-itsr^k\big)\cdot pq^k+
\sum_{k=0}^\infty\big(\exp\{itsr^k\}-1\big)\cdot pq^k \\
&=&\sum_{k=-\infty}^{\infty}
\big(\exp\{itsr^k\}-1-itsr^kc_k\big)\cdot pq^k
\ttt{with}c_k=0\mbox{ for } k\geq0\mbox{ and }c_k=1\mbox{ for } k<0.
\eeaa
\emph{(ii)} If $r> 1/q$, then, for $u>0$,
\[\frac{S_{uM}}{N}\dto Z(u)\ttt{as}\nifi\,, \]
or, equivalently,
\[\frac1{(rq)^{n}}\cdot\frac{S_{uM}}{M}\dto Z(u)\ttt{as}\nifi\,, \]
where now $Z(u)$ is defined via the characteristic function
$\varphi_{Z(u)}(t)=\exp\{ug(t)\}$ with
\[ 
g(t)=\sum_{k=-\infty}^\infty
\big(\exp\{itsr^k\}-1\big)\cdot pq^k\,.
\]\end{theorem}

In complete analogy with \cite{aml85} we infer that 
the limit law is infinitely divisible, that the corresponding L\'evy measure
has point masses $pq^k$ at the points $sr^k$ for
$k\in\mathbb{Z}$, and that we are facing a compound Poisson
distribution with (two-sided geometric weights).

Proofs of Theorems \ref{thm1} and \ref{thm2}, will be given in Sections 
\ref{pfthm1} and \ref{pfthm2}, respectively.

In addition, by replacing $2^m$ by $q^{-m}=r^m$ in the proof of
\cite{aml85}, Theorem 2, it follows immediately
that the limit distribution in
Theorem \ref{thm2}(i) is \emph{semistable\/} in the sense of L\'evy:
\begin{lemma}\label{semi} We have
\[g(t)=q^m\big(g(tq^m)+itspm\big)\ttt{for all}m\in\mathbb{Z}.
\] 
\end{lemma}
In particular, this illustrates the fact that we do not have a limit
distribution for the full sequence (since such a limit would have been
stable with index 1). For more on semistable distributions, cf.\ e.g.\
\cite{meerscheff, sato}. 

\section{Preliminaries}
\setcounter{equation}{0} \label{prel} In this section we collect some
facts that will be used later with or without specific reference.

The following well-known relation holds between logarithms with bases
$r$ and $u$ for $y>0$:
\bea\label{logmojs}\lr{y}=\log_u(y)\cdot \lr{(u)}.\eea

\begin{lemma}\label{lemmax}
For $X$ as defined in Theorem \ref{thm1} we have
\bea
E\big(X^\beta\big)&=&\begin{cases}
\dfrac{s^\beta p}{1-r^{\beta}q},&
\ttt{for}r<q^{-1/\beta}\,,\label{mombeta}\\[3mm]
=\infty,&\ttt{for}r\geq q^{-1/\beta}\,.\label{momb}\end{cases}
\eea
Moreover, as  $ x\to\infty$,
\bea\label{tail}
P(X>x)&=& q^{[\lr{(x/s)}]+1}\geq q^{\lr{(x/s)}+1}\geq \frac{s}{rx}
\ttt{as}x\to\infty\,,\\
E\big(XI\{X\leq x\}\big)&\sim&sp\lr{(x/s)}\ttt{for}r=1/q\,.\label{extrunc}
\eea
\end{lemma}
\begin{proof} Relation (\ref{mombeta}) follows via 
\[
E\big(X^\beta\big)=\sum_{k=1}^\infty (sr^{k-1})^{\beta} pq^{k-1},\]
and the tail estimate is equivalent to formula (1) in \cite{sandor07}. 
The final inequality there exploits the fact that $rq\geq1$.

The asymptotics for the truncated first moment follows via
\[
E\big(XI\{X\leq x\}\big)=\sum_{\{k: s\rkm\leq x\}}s\rkm pq^{k-1}
\sim sp\sum_{1\leq k\leq \lr{(x/s)}+1}1\,.\]\vsp
\end{proof}  

\section{Proof of Theorem \ref{thm1}}
\label{pfthm1}\setcounter{equation}{0} 
Recall that $r=1/q$. We first observe that the function $x\lqq{x}\in\rve$
(that is, regularly varying with exponent 1).

Next, since by (\ref{tail}),
\[
nP(X>n\lr{n})=n\cdot q^{[\alpha\lr{(sn\lr{n})}]+1}
 \sim {n\lr{n}/s}^{-\lr{(1/q)}}\to0\ttt{as}\nifi\,,
\]
and, by (\ref{extrunc}),
\[E\big(XI\{X\leq n\lr{n}\}\big)\sim sp\cdot\lr{(sn\lr{n}n)},
\]
so that 
\[\frac{n\cdot E\big(XI\{X\leq n\lr{n}\}\big)}{n\lr{n}}\to0\ttt{as}\nifi,\]
the conclusion is an immediate consequence of the extension of
Feller's weak law of large numbers given in \cite{g04}, Theorem 1.3;
cf.\ also \cite{g13}, Theorem 6.4.2.\vsb

\section{Proof of Theorem \ref{thm2}}
\label{pfthm2}\setcounter{equation}{0} 
Theorem \ref{thm2}(i) is proved via a fairly straightforward
modification of the corresponding proof in \cite{aml85}.

\subsection*{Proof of (i)}
Since, $P(X=sr^{k-1})=pq^{k-1}$, we have
\[\varphi_X(t)=E\big(e^{itX}\big)
=\sum_{k=1}^\infty e^{itsr^{k-1}}\cdot pq^{k-1},\]
from which it follows that
\beaa
\varphi_{\frac{S_{uN}}{N}-uspn}(t)&=&e^{-ituspn}\Big(\sum_{k=0}^\infty 
e^{i\frac{t}{N}sr^k}\cdot pq^k\Big)^{uN}
=e^{-ituspn}\Big(\sum_{k=0}^\infty e^{itsr^{k-n}}\cdot pq^k\Big)^{uN}\\
&=&e^{-ituspn}\Big(1+\sum_{k=0}^\infty 
\big(e^{itsr^{k-n}}-1\big)\cdot pq^k\Big)^{uN}\\
&=&e^{-ituspn}\Big(1+q^n\sum_{k=-n}^\infty 
\big(e^{itsr^k}-1\big)\cdot pq^k\Big)^{uN}\\
&=&e^{-ituspn}\Big(1+\frac1{N}\sum_{k=-n}^\infty 
\big(e^{itsr^k}-1\big)\cdot pq^k\Big)^{uN}\\
&=&e^{-ituspn}\Big(1+\frac1{N}\sum_{k=-n}^{-1} 
\big(e^{itsr^k}-1 -itsr^k\big)\cdot pq^k \\&&\hskip4pc + itsp\frac{n}{N}
+\frac1{N}\sum_{k=0}^{\infty}
\big(e^{itsr^k}-1\big)\cdot pq^k\Big)^{uN}\,,\\
&=&e^{-ituspn}\bigg(1+\frac1{N}\Big\{\sum_{k=-n}^{-1} 
\big(e^{itsr^k}-1 -itsr^k\big)\cdot pq^k \\&&\hskip4pc + itspn
+\sum_{k=0}^{\infty}\big(e^{itsr^k}-1\big)\cdot pq^k\Big\}\bigg)^{uN}\,,
\eeaa
which converges to $e^{ug(t)}$  as $\nifi$. 
\vsb

\subsection*{Proof of (ii)}
The same computations with obvious modifications yield
\beaa
\varphi_{\frac{S_{uM}}{N}}(t)
&=&\Big(\sum_{k=0}^\infty e^{i\frac{t}{N}sr^k}\cdot pq^k\Big)^{uM}\\
&=&\Big(1+\sum_{k=0}^\infty 
\big(e^{itsr^{k-n}}-1\big)\cdot pq^k\Big)^{uM}\\
&=&\Big(1+q^n\sum_{k=-n}^\infty 
\big(e^{itsr^k}-1\big)\cdot pq^k\Big)^{uM}\\
&=&\Big(1+\frac{1}{M}\sum_{k=-n}^\infty 
\big(e^{itsr^k}-1\big)\cdot pq^k\Big)^{uM}\,,
\eeaa
which converges to $e^{ug(t)}$  as $\nifi$. 
\vsb

\section{How much does one gain until ``game over''? }
\setcounter{equation}{0} \label{stopp}
This section extends results from \cite{aml08}, where the classical
game was treated.

We consider a truncated version of the game in which the duration $T$
of a single game is truncated to $T_n=T\wedge n\leq n$, that is,
``game over'' happens when $T>n$ for the first time.  Otherwise the
gain is as before and the game continues. The following result then
holds for the total gain during one such sequence of games.
\begin{theorem} Let $G_n$ be the total gain until game over, and 
$E$ be a standard exponential random variable.\\\noindent 
\emph{(i)} If\/ $rq=1$, then
\[r^{-n}G_n=q^nG_n\dto qs(E-1)\ttt{as}\nifi.\]
\emph{(ii)} If\/ $rq>1$, then
\[r^{-n}G_n\dto \frac{ps}{r-1}(E-1)\ttt{as}\nifi.\]
\end{theorem}
\begin{remark}\emph{For $r=2$ (i) turns into $2^{-n}G_n\dto\Exp(1)$,
which, if, in addition, $ps=1$ (and, hence, $p=q=1/2$),
reduces to Martin-L\"of's Theorem 2.1.}
\end{remark}
\begin{remark}\emph{Note that if we, formally, set $rq=1$ in (ii), 
then (ii) reduces to (i).}\vsb
\end{remark}

\begin{proof}
Let  $N_n =$ the number of rounds until game over. The
first observation then is that, since $P(T>n)=q^n$, it follows that  
$N_n$ has a geometric distribution with mean $q^{-n}$.

The truncated gain is given by
\beaa
P(X_n=s\rkm)&=&P(T=k)=p\qkm\ttt{for}k\leq n\\[4pt]
P(X_n=0)&=&P(T>n)=q^n
\eeaa
Since the fee is $ps\rkm$ in round $k$,
the net gain, that is, the true gain $-$ the amount spent, becomes
\[
V_n=
\begin{cases}s\rkm-p\big(s+sr+\cdots+s\rkm\big)\\
=s\rkm-ps\dfrac{r^k-1}{r-1}=s\rkm\dfrac{qr-1}{r-1}+\dfrac{ps}{r-1},
&\ttt{if}T=k\leq n\\[5mm]
0-p\big(s+sr+\cdots+sr^{n-1}\big)=-ps\dfrac{r^n-1}{r-1},
&\ttt{if}T>n.\end{cases}
\]
It is now easy to check that $E\,V_n=0$, so the game is fair.

\subsection*{Proof of (i)}
Now, suppose that $rq=1$. Then
\[
V_n=
\begin{cases}\dfrac{ps}{r-1}=qs,&\ttt{if}T=k\leq n\\
-qsq^{-n}+qs=qs(1-q^{-n}),&\ttt{if}T>n,\end{cases}
\]
which tells us that the total gain until ``game over'' equals
\[G_n=qs\cdot (N_n-1) +qs(1-q^{-n})=qs\cdot(N_n-q^{-n})\]
Furthermore, since, as noted above, $N_n$ has a geometric distribution with 
mean $q^{-n}$, it is well-known that
\[q^nN_n\dto\Exp(1)\ttt{as}\nifi,\]
from which the conclusion follows.

\subsection*{Proof of (ii)}
This case is a bit harder, since $G_n$ now is equal to a sum of $N_n-1
$ \iid random variables corresponding to gains, thus distributed as
$V_n^{+}$, say, and one final ``game over''-variable, distributed as
$V_n^-$, say. All summands are independent of $N_n$. This thus allows
us to resort to the well-known relation for the characteristic
function of a sum of a random number of \iid random variables, which
in our case amounts to
\bea \label{SN}
\varphi_{G_n}(t)=g_{N_n-1}\big(\varphi_{V_n^+}(t)\big)\cdot\varphi_{V_n^-}(t) 
\,\eea
where $\varphi$ and $g$ denote characteristic and (probability) generation 
functions, respectively.

As for $N_n-1$, we have
\[g_{N_n-1}(t)=\frac{q^n}{1-(1-q^n)t}.\]
Furthermore,
\beaa
\varphi_{V_n^+}(t)&=&\sumk\exp\Big\{it\cdot\Big(\frac{s(qr-1)}{r-1}\cdot
\rkm+\frac{ps}{r-1}\Big)\Big\}\cdot\frac{pq^{k-1}}{1-q^n}\\
&\sim&1+\sumk it\cdot\frac{s(qr-1)}{r-1}\cdot
\rkm\cdot\frac{pq^{k-1}}{1-q^n}+it\cdot\frac{ps}{r-1}\\
&=&1+it\cdot\frac{s(qr-1)}{r-1}\cdot\frac{p}{1-q^n}
\sum_{k=0}^{n-1}\big(qr\big)^k+it\cdot\frac{ps}{r-1}\\
&=&1+it\cdot\frac{ps(qr-1)}{(r-1)(1-q^n)}
\cdot\frac{(qr)^n-1}{qr-1}+it\cdot\frac{ps}{r-1}\\[2mm]
&=&1+it\cdot\frac{psq^n(r^n-1)}{(r-1)(1-q^n)}
\sim 1+it\cdot\frac{ps}{r-1}\cdot(qr)^n
\,,
\eeaa
and
\[
\varphi_{V_n^-}(t)= \exp\big\{-itps\cdot\frac{r^n-1}{r-1}\big\}\,.\]
An application of (\ref{SN}) therefore tells us that
\beaa\varphi_{G_n}(t)&\sim& 
\frac{q^n}{1-(1-q^n)\big(1+it\cdot\frac{ps}{r-1}\cdot(qr)^n\big)}
\cdot\exp\big\{-itps\cdot\frac{r^n-1}{r-1}\big\} \\
&=&\frac{1}{1-it(1-q^n)\frac{ps}{r-1}r^n}
\cdot\exp\big\{-itps\cdot\frac{r^n-1}{r-1}\big\}\,,
\eeaa
and, hence, that
\[
\varphi_{r^{-n}G_n}(t)
\sim \frac{1}{1-it(1-q^n)\frac{ps}{r-1}}
\cdot\exp\big\{-itps\cdot\frac{1-r^{-n}}{r-1}\big\}
\to\frac{1}{1-it\frac{ps}{r-1}}\cdot e^{-it\frac{ps}{r-1}}\ttt{as}\nifi\,,
\]
which, in view of the continuity theorem for characteristic functions, 
finishes the proof of (ii).
\vsb\end{proof}

\section{Capital with interest}
\setcounter{equation}{0} \label{rta} 
Following \cite{aml08} in this section we assume that the player can
borrow money without restriction and that he has to pay interest on
the capital with a discount factor $\gamma<1$ per game. Once again we
consider the model (\ref{jag}), where now $r=1/q$, introducing $T$ as
the generic duration of a single game, viz.,
\bea\label{duration}
P(T=k)=P(X=s\rkm)=p\qkm,\quad k=1,2,\ldots.
\eea
In this case the present value of the gain equals $\gamma^TX$, which 
has finite expectation;
\bea\label{value}
E\big(\gamma^TX\big)=\sum_{k=1}^\infty \gamma^ks\rkm p\qkm
=\frac{sp\gamma}{1-\gamma}<\infty,
\eea
(which reduces to Martin-L\"of's $\gamma/(1-\gamma)$ when $p=q=1/2$ and 
$ps=1$).

If an infinite number of games are played they occur at times
$T_1,\,T_2,\,\ldots$ forming a renewal process with increments
$\tau_k=T_k-T_{k-1}$, $k\geq1$ (with $T_0=0$) having the same
distribution as $T$. The present value of the total gain is then given by
\bea\label{vinst}
V(\gamma)=\sum_{k=1}^\infty \gamma^{T_k}X_k
= \sum_{k=1}^\infty \gamma^{T_{k-1}}\gamma^{\tau_k}sr^{\tau_k-1}.\eea 
We now want to find an asymptotic distribution of $V(\gamma)$ when
$\gamma\nearrow1$. As in \cite{aml85} we scale time by a factor
$N=r^n=q^{-n}$ (cf.\ Theorem \ref{thm2}(i)). The renewal process
$\{T_k,\,k\geq1\}$ then has a deterministic limit
\[\frac{T_{uN}}{N}\asto uE\,T=\frac{u}{p}\ttt{as}N\to\infty,\]
and 
\[\frac{S_{uN}}{N}-uspn=\frac1{N}\sum_{k=1}^{uN}(X_k-spn)
\dto Z(u) \ttt{as}N\to\infty,\]
for fixed $u>0$, and where $\{Z(u),\,u\geq0\}$ is the L\'evy process defined 
via the characteristic function
\[\varphi_{Z(u)}(t)=E\Big(e^{itZ(u)}\Big)=e^{ug(t)},\] 
where, in turn, $ug(t)$ is the  L\'evy exponent with
\bea
g(t)&=&\sum_{k=-\infty}^{-1}
\big(\exp\{itsr^k\}-1-itsr^k\big)\cdot pq^k+
\sum_{k=0}^\infty\big(\exp\{itsr^k\}-1\big)\cdot pq^k\nonumber\\
&=&\sum_{k=-\infty}^{\infty}
\big(\exp\{itsr^k\}-1-itsr^kc_k\big)\cdot pq^k \label{kf}\,,
\eea
where $c_k=0$ for $k\geq0$, and $c_k=1$ for $k<0$.

It follows that
\[\frac1NV(\gamma)=\frac1N\sum_{k=1}^\infty\gamma^{T_k}(X_k-spn)
+\frac1N\sum_{k=1}^\infty\gamma^{T_k}spn,\]
and, setting $\gamma=\exp\{-ap/N\}$, we obtain
\[\frac1NV(\gamma)-\frac{spn}{N}\sum_{k=1}^\infty e^{-apT_k/N}
=\frac1N\sum_{k=1}^\infty e^{-apT_k/N}(X_k-spn).
\]
Letting $N\to\infty$ yields
\[
\frac1NV(\gamma)-spn\int_0^\infty e^{-au}\,du\dto\int_0^\infty e^{-au}\,dZ(u),
\] 
i.e.,
\bea\label{limit}
\frac1NV(\gamma)-\frac{spn}{a}\dto\int_0^\infty e^{-au}\,dZ(u)\,.
\eea
This is interesting because of the following 
\begin{lemma}\label{lemma71} 
The characteristic function of the random variable
\[U=\int_0^\infty  e^{-au}\,dZ(u)\]
equals
\[\varphi_U(t)=E\big(e^{itU}\big)=e^{g(t)/a},\]
with
\[g(t)=itsq+\sum_{k=-\infty}^\infty
\int_{s\rkm}^{s\rk}q^k\big(e^{itx}-1-itxc_k\big)
\frac{dx}{x}.\]
The L\'evy measure thus has a density 
\[q^k\frac{dx}{x}\ttt{for}s\rkm<x\leq s\rk\ttt{and all}k.\]
\end{lemma}
\begin{proof} Exploiting formula (\ref{kf}) we find that
\[Z(u)=\sum_{k=-\infty}^\infty s\rk Z_k(u),\]
where $\{ Z_k(u)\}$ are independent having characteristic function
\[\varphi_{Z_k}(u)=\exp\{upq^k(e^{it}-1-itc_k)\}.\]
This tells us that
\[U=\sum_{k=-\infty}^\infty s\rk U_k\ttt{with} U_k
=\int_0^\infty e^{-au}\,dZ_k(u)\,.\]
Since  $\{Z_k(u),\,u\geq0\}$ has independent increments for all $k$, this 
means, via a change of variable, that
\beaa
\varphi_{U_k}(t)
&=&\exp\Big\{\int_0^\infty pq^k\big(e^{ite^{-au}}-1-ite^{-au}c_k\big)\,du
\Big\}\\
&=&\exp\Big\{\int_0^1 \frac{pq^k}{a}\big(e^{itx}-1-itxc_k\}\big)\frac{dx}{x}
\Big\}\,,
\eeaa
and, hence, that
\beaa
\varphi_{U_k}(itsr^k)
&=&\exp\Big\{\int_0^{sr^k}\frac{pq^k}{a}\big(e^{itx}-1-itxc_k\big)\frac{dx}{x}
\Big\}\\
&=&\exp\Big\{\sum_{j=-\infty}^k\int_{sr^{j-1}}^{sr^j} 
\frac{pq^k}{a}\big(e^{itx}-1-itxc_k\big)\frac{dx}{x}\Big\}\,.
\eeaa
Summing over $k$ we then obtain 
\bea\label{kfu}
\varphi_U(t)=\prod_{k=-\infty}^\infty \varphi_{U_k}(its\rk)
=\exp\Big\{\sum_{k=-\infty}^\infty\sum_{j=-\infty}^k\int_{sr^{j-1}}^{sr^j} 
\frac{pq^k}{a}\big(e^{itx}-1-itxc_k\big)\frac{dx}{x}\Big\}\,.
\eea
Now, for fixed $j\geq0$ we have $\sum_{k=j}^\infty q^k=q^j/p$, so that
\[\sum_{k=j}^\infty c_kq^k
=\begin{cases} 0=c_j,&\ttt{for}j\geq0,\\
\sum_{k=j}^{-1} c_kq^k=\dfrac1{q}\cdot\dfrac{(1/q)^{|j|}-1}{(1/q)-1}
=\dfrac{q^{j}-1}{p}\cdot c_j,&\ttt{for}j<0.\end{cases}
\]
Inserting this into (\ref{kfu}), and changing the order of summation, 
finally shows that
\beaa
\varphi_U(t)
&=&\exp\Big\{\sum_{j=-\infty}^\infty\int_{sr^{j-1}}^{sr^j}
\Big(\frac{q^j}{a}\big(e^{itx}-1\big)-\frac{itxc_j}{a}(q^j-1)\Big)\frac{dx}{x}
\Big\}\\
&=&\exp\Big\{\sum_{j=-\infty}^\infty\int_{sr^{j-1}}^{sr^j}
\Big(\frac{q^j}{a}\big(e^{itx}-1-itxc_j\big)\frac{dx}{x}
+\frac{itc_j}{a}\,dx\Big)\Big\}\\
&=&\exp\Big\{\sum_{j=-\infty}^\infty\int_{sr^{j-1}}^{sr^j}
\frac{q^j}{a}\big(e^{itx}-1-itxc_j\big)\frac{dx}{x}
+\sum_{j=-\infty}^{-1}\big(sr^j-sr^{j-1}\big)\frac{it}{a}\Big\}\\
&=&\exp\Big\{\sum_{j=-\infty}^\infty\int_{sr^{j-1}}^{sr^j}
\frac{q^j}{a}\big(e^{itx}-1-itxc_j\big)\frac{dx}{x}
+\frac{itqs}{a}\Big\}=e^{g(t)/a}\,.
\eeaa
 \vsp
\end{proof}

\subsection{$U$ is semistable}
Next we prove an analog of Lemma \ref{semi}, to the effect that the 
distribution of $U$ is semistable.
\begin{lemma}\label{semistable} For any integer $m$ we have
\[g(tq^m)=q^m\big(g(t)+itsmp\big)\,.\]
\end{lemma}
\begin{proof} We first observe that
\beaa
g(tq^m)&=&(itsq)q^m+\sum_{k=-\infty}^\infty
\int_{s\rkm}^{s\rk}q^k\big(e^{itq^mx}-1-itq^mxc_k\big)\frac{dx}{x}\\
&=&(itsq)q^m+\sum_{k=-\infty}^\infty
\int_{sr^{k-m-1}}^{sr^{k-m}}q^k\big(e^{itx}-1-itxc_k\big)\frac{dx}{x}\\
&=&(itsq)q^m+\sum_{k=-\infty}^\infty
\int_{s\rkm}^{s\rk}q^{k+m}\big(e^{itx}-1-itxc_{k+m}\big)\frac{dx}{x}\\
&=&q^m\Big(g(t)+\sum_{k=-\infty}^\infty
\int_{s\rkm}^{s\rk}q^kitx(c_k-c_{k+m})\,\frac{dx}{x}\Big)\\
&=&q^m\Big(g(t)+\sum_{k=-\infty}^\infty
(s\rk-s\rkm)q^kit(c_k-c_{k+m})\Big)\\
&=&q^m\Big(g(t)+\sum_{k=-\infty}^\infty
itsp(c_k-c_{k+m})\Big)\\[4pt]
&=&q^m(g(t)+itspm).
\eeaa
\vsp\end{proof}

\subsection{The tail of $U$}
Our next step is to exploit the semistability for
an estimate for the tail of the distribution of $U$.

Toward that end, set $\bar U_m=q^m(U-pms/a)$ and let $\varphi_m$ be the 
characteristic function of $U_m$, viz.,
\beaa
\varphi_m(t)&=&e^{-itq^mpms/a}\varphi_U(q^mt)
=\exp\big\{ - itq^mpms/a+q^mg(t)/a\big\}\\[3pt]
&=&\exp\big\{q^mg(t)/a\big\} \approx 1+q^mg(t)/a\ttt{for $m$ large,}
\eeaa
from which we conclude that
\bea\label{semisvans}
q^{-m}(\varphi_m(t)-1)\to g(t)/a\ttt{as}m\to\infty.
\eea
Now, the LHS equals the L\'evy exponent corresponding to the L\'evy measure
$L_m(dx)=q^{-m}\cdot P(\bar U_m\in dx)$ and the RHS has L\'evy measure $L(dx)$. Using the continuity theorem
for L\'evy exponents, cf.\ \cite{FII}, Chapter XVII.2, Theorem 2,
we thus conclude that
\bea\label{lbar}
q^{-m}P(U_m>x)\to \int_x^\infty L(dy) = \bar L(x)\ttt{as}m\to\infty;\quad(x>0),
\eea
and, hence, that
\bea\label{usvans}
q^{-m}P(U>xq^{-m}+pms/a)\to\bar L(x)\ttt{as}m\to\infty;\quad(x>0).
\eea
From Lemma \ref{lemma71} we remember that $L$ has density 
\[q^k\frac{dx}{x}\ttt{for}s\rkm<x\leq s\rk\ttt{and all}k,\]
from which we infer that
\bea
\bar L(x\rk)&=&\frac{q^k}{a}\int_{x\rk}^{s\rk}\frac{dx}{x}
+\sum_{j=k+1}^\infty\frac{q^j}{a}\int_{sr^{j-1}}^{sr^j}\frac{dx}{x}
\nonumber\\[3pt]
&=&\frac{q^k}{a}\log(s/x)+\sum_{j=k+1}^\infty\frac{q^j}{a}\log r\nonumber\\[3pt]
&=& \frac{q^k}{a}\log(s/x)+\frac{q^{k+1}}{pa}\log r\nonumber\\[3pt]
&=& \frac{q^k}{a}\big(\frac{q}{p}\log r -\log(x/s)\big)\ttt{for}qs<x<s.
\label{Lsvans}
\eea

\subsection{An infinite number of games}
Let us now see how this can be used to analyze an infinite number of
St.\ Petersburg games.

Consider first a single game and put, for simplicity, $s=r=1/q$, so
that $X=r^T$. The fee for playing round $k$ in one game then is
$\gamma pr^k$ (= the present value of the stake prior to round $k$).
The present value of the net gain at the beginning of the game then,
recalling that $rq=1$, becomes
\beaa
&&\hskip-4pc\gamma^Tr^T
-\gamma p\cdot\big(r+\gamma r^2+\gamma^2r^3+\cdots+\gamma^{T-1}r^T\big)\\[3pt]
&=&(\gamma r)^T-\gamma pr\sum_{j=0}^{T-1}(\gamma r)^j\big)\\[3pt]
&=&(\gamma r)^T-\gamma pr\cdot\frac{(\gamma r)^T-1}{\gamma r-1}\\[3pt]
&=&(\gamma r)^T\cdot\Big(1-\frac{\gamma pr}{\gamma r-1}\Big)
+\frac{\gamma pr}{\gamma r-1}\\[3pt]
&=&\frac{\gamma pr}{\gamma r-1}-\frac{1-\gamma}{\gamma r-1}\cdot(\gamma r)^T.
\eeaa
Since $T$ has a geometric distribution with mean $1/p$ we conclude that
the expected value of this quantity eqals
\[
\frac{\gamma pr}{\gamma r-1}-\frac{1-\gamma}{\gamma r-1}\cdot E(\gamma r)^T
=\frac{\gamma pr}{\gamma r-1}
-\frac{1-\gamma}{\gamma r-1}\cdot\frac{p\gamma r}{1-\gamma}=0\,,
\]
which tells us that the game is fair.

In analogy with (\ref{vinst}) we thus conclude that the present value of the 
total gain equals
\[
\tilde V=\sum_{k=1}^\infty \gamma^{T_{k-1}}\cdot
\Big(\frac{\gamma pr}{\gamma r-1}
-\frac{1-\gamma}{\gamma r-1}\cdot(\gamma r)^{\tau_k}\Big).
\]
The asymptotic expansion with $\gamma=e^{-ap/N}$ and $N=r^n$, so that
$1-\gamma\sim ap/n$, $\gamma r-1\sim r-1 = (1/q)-1=p/q=pr$, then turns this into
\[
\tilde V\approx \sum_{k=1}^\infty  \gamma^{T_i}-\frac{aq}{N}V(\gamma)
\approx N\int_0^\infty e^{-au}\,du-aqU
=\frac{N}{a}-aq\big(U+\frac{spn}{a}\big)\,.
\]
This tells us that, neglecting $spn/a$, the ruin probability $P(\tilde
V<0)$ can be approximated by $P(U>N/(qa^2))$.

Finally, by exploiting formulas (\ref{usvans}) and (\ref{Lsvans}) 
concerning the tail of the distribution of $U$, recalling that $s=r$, we
obtain the following approximation for this probability:
\bea
P(\tilde V<0)&\approx& P\Big(U>\frac{N}{qa^2}\Big)\approx
 P\Big(U>\frac{N}{qa^2}+\frac{\log N}{r}\Big)\nonumber\\
&\approx&\frac1N\bar L(x)\approx
\frac1{aN}\Big(\frac{q}{p}\log r-\log(x/r)\Big)\nonumber\\
&=&\frac1{aN}\Big(\frac{1}{p}\log r-\log x\Big)
\ttt{for}x=\frac1{qa^2}\mbox{ and }1<x<r.\label{ruin}
\eea
For the special case $p=q=1/2$ the result reduces to that of 
Martin-L\"of, \cite{aml08}.

\section{Some remarks}
\setcounter{equation}{0} \label{anm}
We close with some additional results and comments.

\subsection{Polynomial and geometric size deviations}
\label{stoica} In this subsection we provide
immediate extensions of the results from \cite{g10}, Section 7 (cf.\
also \cite{g10c}), which, in turn were inspired by \cite{hunyr} and
\cite{stoica08} respectively.

Theorem 2.1 and Corollary 2.3 of Hu and Nyrhinen \cite{hunyr} adapted to 
the present setting yield the following result.

\begin{theorem}\label{thmhunyr} For any $b>1$,
\beaa
\lim_{\nifi}\frac{\lr{P(S_n>n^b)}}{\lr{n}}&=&1-b\,,\\
\lim_{\nifi}\frac{\lr{P(M_n>n^b)}}{\lr{n}}&=&1- b\,.
\eeaa
\end{theorem}

\begin{proof} The only thing to check is, in the notation of \cite{hunyr}, 
that $\overline{\alpha}=\underline{\alpha}$ (with $\alpha=1$)
in formulas (5) and (6) there, and this is immediate, since, by (\ref{tail}),
\[P(\log X>x)=P(X>e^x)\to-1\ttt{as}x\to\infty.\]
\vsp\end{proof}
As for geometric size deviations, we have

\begin{theorem} \label{thmstoica} \emph{(i)} For
any $\varepsilon>0$ and $b>1$,
\bea\label{x}
\lim_{\nifi}\frac{\lr{P(X>\varepsilon b^n)}}
{\lr{(\varepsilon b^n)}}=-1\,.\eea
\emph{(ii)} Suppose, in addition, that $E(X^{1/b})<\infty$, for some
$b>(\lr{1/q})^{-1}$ $(\geq1)$.
Then
\bea
\label{sum}
\lim_{\nifi}\frac{\lr{P(S_n>\varepsilon b^n)}}{n}
&=&-\lr{b}\,,\\
\label{max}
\lim_{\nifi}\frac{\lr{P(M_n>\varepsilon b^n)}}{n}
&=&-\lr{b}\,.
\eea
If, in particular, $b=r=1/q$ the limits in (\ref{sum}) and (\ref{max})
equal\/ $-1$.
\end{theorem}
\begin{proof} Relation (\ref{x}) is an immediate consequence of (\ref{tail}), 
and for (\ref{max}) we exploit \cite{g13}, Lemma 4.2, to conclude that 
\bea\label{maxx} 
\frac12 nP\big(X>\varepsilon b^{n}\big)
\leq P\big(M_n>\varepsilon b^{n}\big)
\leq nP\big(X>\varepsilon b^{n}\big)\ttt{for $n$ large,}
\eea 
cf.\ 
(cf.\ \cite{g13}, p.\ 270), after which the remaining details are the same
as in \cite{g10,g10c}.\vsb\end{proof}

\subsection{Almost sure convergence?}
 In this subsection we discuss possible almost sure
 convergence in Theorem \ref{thm1}.

Now, since $E\,X=+\infty$, the converse of the Kolmogorov
 strong law provides a negative answer. However, more can
 be said.  Namely, since $S_n\geq X_n$ for all $n\geq1$, it follows, via
(\ref{tail}), that
\[
\sumin P(S_n>cn\lr{n}\geq \sumin P(X>cn \lr{n})
\geq\sumin\frac{s}{cn\lr{n}}=\infty\ttt{for any}c>0.\]
The first Borel--Cantelli lemma therefore tells us that 
$P(S_n>cn\lr{n}\mbox{ i.o.})=1$ for any $c>0$, and, hence, that
\bea\label{limsup} \limsup_{\nifi}\frac{S_n}{n\lr{n}}=+\infty.  \eea
As for the limit inferior, following \cite{adler}, we set
$\mu(x)=\int_0^xP(X>y)\,dy$, and note, via partial integration and
Lemma \ref{lemmax}, that
\[\mu(x)=xP(X>x)+ \int_0^xydF_X(y)=q^{[\lr{(x/s)}]+1}+sp\lr{(x/s)}
\sim sp\lr{(x/s)}\ttt{as}x\to\infty,
\]
from which we, via (\ref{logmojs}), also conlcude that $\mu(x)\sim
\mu(x\lz{x})$ as $x\to\infty$. An application of \cite{adler},
Theorem 2 (with $\alpha=0$ and $b_n=n\lr{n}$), therefore asserts that
that
\[\liminf_{\nifi}\frac{S_n}{n\lr{n}}=ps.\]
For the case $ps=1$ the conclusion obviously reduces to Example 4 of 
\cite{adler}, cf.\ also \cite{cs96}. 

\subsection*{Acknowledgement} 
We wish to thank Professor Toshio Nakata for his careful reading of
our manuscript and his remarks that clarified some obscurities and 
helped us to improve the paper.

\medskip\noindent {\small Allan Gut, Department of Mathematics,
Uppsala University, Box 480, SE-751\,06 Uppsala, Sweden;\newline 
Email:\quad \texttt{allan.gut@math.uu.se}\qquad
URL:\quad \texttt{http://www.math.uu.se/\~{}allan}}

\smallskip\noindent {\small Anders Martin-L\"of, Department of Mathematics,
Stockholm University, SE-106\,91 Stockholm, Sweden;\newline 
Email:\quad \texttt{andersml@math.su.se}\qquad
URL:\quad \texttt{http://www2.math.su.se/\~{}andersml}}
\end{document}